\documentclass[a4paper]{amsart}
\usepackage{amssymb,amsmath,float,amsfonts,latexsym,amsthm,comment,pinlabel,comment,lineno}
\numberwithin{equation}{section}

\usepackage[mathscr]{eucal}

\usepackage{float}
\usepackage[T1]{fontenc}
\usepackage{txfonts}
\usepackage[pdftex,bookmarks=true]{hyperref}
\usepackage{graphicx}
\usepackage{subfigure}
\usepackage{wrapfig}
\usepackage{import}
\usepackage{epsfig}
\usepackage{enumerate}
\usepackage[all]{xy}
\usepackage{booktabs}
\usepackage{xcolor}
\theoremstyle{plain}
\newtheorem{thm}[subsection]{Theorem}

\newtheorem{defn}[subsection]{Definition}

\newtheorem{lemma}[subsection]{Lemma}

\newtheorem{prop}[subsection]{Proposition}
\theoremstyle{definition}

\newtheorem{rmk}[subsection]{Remark}
\newtheorem{eg}[subsection]{Example}

\numberwithin{equation}{section}

\author{Subash Chandra Behera}
\address{
School of Mathematics \& Computer Science\\
Indian Institute of Technology Goa\\
At Goa College of Engineering Campus\\
Farmagudi, Ponda-403401 \\
Goa, India}
\email{subash20232102@iitgoa.ac.in}
\author{Shiv Parsad}
\address{
School of Mathematics \& Computer Science\\
Indian Institute of Technology Goa\\
At Goa College of Engineering Campus\\
Farmagudi, Ponda-403401 \\
Goa, India} 
\email{shiv@iitgoa.ac.in}
\begin{document}

\title{Minima of geodesic length functions for non-uniform filling systems}



\begin{abstract}  

Let $S_g$ be a closed oriented surface of genus $g\geq 2$. Kerckhoff proved that the geodesic length function $\ell_\Omega$ of a filling $\Omega$ on $S_g$ attains a unique minimum in Teichmüller space \cite{KSP}. Recently, Girondo et al. computed these minima for uniform fillings using the algebraic machinery of \textit{dessins d'enfants} and Grothendieck-Belyi surfaces \cite{GEGGHR}. We present an elementary optimization approach for $4$-regular topological uniform fillings, bypassing this framework. Furthermore, we analyze two special classes of non-uniform $4$-regular fillings using fat graphs and optimization techniques. We explicitly compute their minima and prove that in both classes, the minimum of these length functions is attained at a triangle surface.
\end{abstract}

\maketitle

%%%%%%%%%%%%%% Section 1 (Introduction) %%%%%%%%%%%%

\section{Introduction}

% Let $S_g$ be a closed oriented surface of genus $g > 2$. A \emph{marking} on $S_g$ is a pair $(f,S_g)$, where $X$ is a closed hyperbolic surface and $f : S_g \rightarrow X$ is an orientation-preserving homeomorphism. Two markings $(f_1,X_1)$ and $(f_2,X_2)$ are considered equivalent if there is an isomorphism $\phi : X_1 \rightarrow X_2$ such that $f_2^{-1} \circ \phi \circ f_1$ is isotopic to the identity map in $S_g$. The set of equivalence classes of markings is the \emph{Teichmüller space} $\mathcal{T}_g$. For a fixed $g>2$, the space of hyperbolic surfaces up to isometry is called the moduli space and is denoted by $\mathcal{M}_g.$

Let $S_{g}$ denote a closed, orientable surface of genus $g\geq 2$. A set $\Omega = \{\gamma_1, \ldots, \gamma_d\}$ of pairwise non-homotopic simple closed curves on $S_{g}$, such that $\gamma_i$ and $\gamma_j$ are in minimal position for all $i$ and $j$, is called a filling system or simply a filling on $S_{g}$ if $S_{g}\setminus\Omega$ is a disjoint union of topological discs. This filling is said to be uniform if all the complementary regions have the same combinatorial length. Originating in the work of Thurston \cite{Thu85}, the theory of filling has grown into a vital tool for understanding mapping class groups, Teichmüller spaces, and moduli spaces of surfaces. Many authors have studied the filling of surfaces (see \cite{ThesisChangHong, RakeshDehnFilling, PasadSankiFillingSystem, BholaSankiFilling,BS,SANKI_2018}).

Foundational work related to geodesic length functions on hyperbolic surfaces  have been studied by Hadamard, Poincaré, Nielsen, Fenchel, and others \cite{FWN,HJ,JN27, Poin1882}. In \cite{KSP}, Kerckhoff proved that a geodesic length function is strictly convex or constant along an earthquake path. In \cite{Wolpert2}, Wolpert shows that a geodesic length function is strictly convex along a Weil–Petersson geodesic. In \cite{Luo}, Luo characterizes the geodesic length functions. In \cite{Wolpert1}, Wolpert discusses the behavior of the geodesic length functions on Teichmüller space and their implications for the Weil–Petersson geometry of Teichmüller space, alongside obtaining bounds for the gradient and Hessian of these functions.

% More recently, Girondo, González-Diez, and Hidalgo achieved the explicit calculation of the minima for these functions in the case of uniform filling using the heavy algebraic machinery of \textit{dessins d'enfants} \cite{GEGGHR}. Saha and Sanki computed the length of filling pairs whose complement consists of exactly two complementary disks \cite{SahaSanki25}. Furthermore, An, Saha, and Sanki classified all minimal filling pairs in genus two and determined the length of the shortest minimal filling pair on $S_2$ \cite{Bhola26}.

 A famous result, first proved by Kerckhoff, states that if $\Omega$ is a filling on $S_g$, then the geodesic length function $\ell_\Omega$ attains a minimum in Teichmüller space, and this minimum is unique \cite{KSP}. The explicit calculation of the minima for uniform filling (allowing arbitrary intersection degrees) was recently achieved using the algebraic and combinatorial machinery of \textit{dessins d'enfants} and Grothendieck-Belyi surfaces \cite{GEGGHR}. By contrast, our focus is specifically on $4$-regular topological fillings.\\ 
 
  In this article, we provide an elementary approach to these calculations through direct optimization techniques, bypassing the heavy algebraic machinery used in previous literature.  Furthermore, An, Saha, and Sanki classified all minimal filling pairs in genus two and determined the shortest minimal filling pair on $S_2$ \cite{Bhola26}. In this article, we study the minima of geodesic length functions for both uniform and non-uniform fillings on $S_g$. The article is structured as follows:

\begin{enumerate}[1.]
\item \textit{Preliminaries:} We introduce the necessary background, notations, and definitions.

\item \textit{Optimization techniques:} We establish Proposition \ref{Proposition} and Theorem \ref{Uniform case}, which form the core optimization framework for the article.

\item \textit{Uniform Case:} We apply this framework to compute the minima for uniform $4$-regular fillings on $S_g$, providing a simpler and more elementary approach than \cite{GEGGHR}.

\item \textit{Non-Uniform Case:} We extend these techniques to two special classes of non-uniform $4$-regular fillings. Using fat graphs, we explicitly compute their minima and prove they are attained at a triangle surface.
\end{enumerate}

\section{Preliminaries}
In this section, we define Teichmüller space as the space of marked hyperbolic surfaces and formally define geodesic length functions. Further, we present hyperbolic trigonometry formulas for a right angled hyperbolic triangle and include expressions for area and perimeter, which are essential for proving our desired results.

\begin{defn}
    A marking on $S_g$ is a pair $(X, f)$, where $X$ is a closed hyperbolic surface and $f : S_g \to X$ is an orientation-preserving homeomorphism.  
    Two markings $(X_1, f_1)$ and $(X_2, f_2)$ are called equivalent if there exists an isometry $\phi : X_1 \to X_2$ such that  
    $f_2^{-1} \circ \phi \circ f_1$ is isotopic to the identity on $S_g$.
\end{defn}

We denote the equivalence class of $(X, f)$ by $[(X, f)]$. It consists of all markings equivalent to $(X, f)$.

\begin{defn}
    The Teichmüller space of $S_g$, denoted $\mathcal{T}_g$, is the set of all such equivalence classes of markings. That is,
    \[
    \mathcal{T}_g = \{\, [(X, f)] \mid (X, f) \text{ is a marking on } S_g \,\}. 
    \]
\end{defn}

Next, we provide a formal definition of the geodesic length function.
\begin{defn}
 For a fixed homotopically nontrivial curve $\gamma$ on $S_g$, the geodesic length function is the map
\[
\ell_\gamma : \mathcal{T}_g \to (0, \infty),
\]
which sends a marked hyperbolic surface $[(f, X)]$ to the length of the unique closed geodesic on $X$ in the homotopy class of $f(\gamma)$. For a finite collection of curves $\Omega = \{\gamma_1, \dots, \gamma_n\}$, the corresponding geodesic length function $\ell_\Omega : \mathcal{T}_g \to (0, \infty)$ is defined as the sum
\[
\ell_\Omega := \sum_{i=1}^n \ell_{\gamma_i}.
\]
    
\end{defn}

Now we present hyperbolic trigonometry formulae for a right-angled hyperbolic triangle in the following lemma.

\begin{lemma}\label{sine and cosine rules} \cite{PB}
Let \( ABC \) be a hyperbolic triangle with side lengths \( a, b, c \), where the side of length \( a \) is opposite to angle \( A \) and there is a right angle at \( A \) (see Figure \ref{RightAngledTriangle}). Then, the following relations hold:

\noindent%
\begin{minipage}{0.5\textwidth}
\begin{enumerate}
\item[(i)]  \(\cosh a=\cosh b \cosh c \),  
\item[(ii)]  \(\cosh a= \cot B \cot C\),
\item[(iii)]  \(\sinh b= \sin B \sinh a\),  
\item[(iv)]  \(\sinh c=\cot B \tanh b\),  
\item[(v)]  \(\cos C =\cosh c \sin B\),  
\item[(vi)]  \(\cos B =\tanh c \coth a\).  
\end{enumerate}
\end{minipage}%
\begin{minipage}{0.3\textwidth}
\hfill%
\begin{figure}[H]
    \centering
    \includegraphics[width=0.7\linewidth]{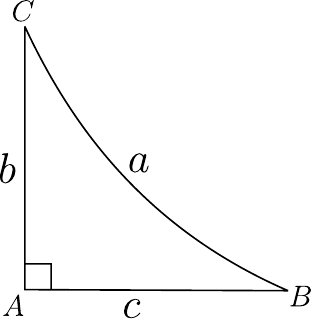}
    \caption{}
    \label{RightAngledTriangle}
\end{figure}
\end{minipage}
\end{lemma}

Next, we present the expression for the perimeter of a regular $n$-gon in terms of its angle, and we state the Gauss--Bonnet theorem.

\begin{prop}
    Let $P$ be a regular hyperbolic $n$-gon with interior angle $\theta$. The perimeter of $P$ is given by
    \(
    \text{Peri}(P) = 2n \cosh^{-1}\left( \frac{\cos(\pi/n)}{\sin\left( \theta/2 \right)} \right).
    \)
\end{prop}
\begin{proof}

\begin{figure}[htbp]
    \centering
    \includegraphics[width=0.3\linewidth]{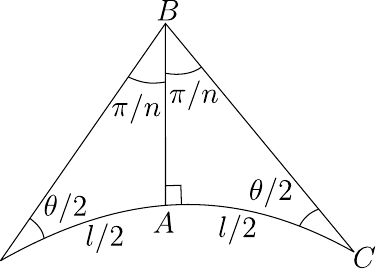}
    \caption{A triangular section of regular hyperbolic $n$-gon.}
    \label{Perimeter}
\end{figure}

Let \( B \) be the circumcenter of the polygon \( P \) and the length of each side of \( P \) is \( 2\ell \). The perpendicular projection of \( B \) onto any side bisects both the side and the angle at \( B \) (see Figure \ref{Perimeter}).

 By Lemma \ref{sine and cosine rules}, we have
   \(
   \frac{\ell}{2} = \cosh^{-1}\left( \frac{\cos(\pi/n)}{\sin(\theta/2)} \right).
   \)
   Therefore, the perimeter of $P$ is given by
   \[
   \text{Peri}(P) = n\ell = 2n \cosh^{-1}\left( \frac{\cos(\pi/n)}{\sin (\theta/2)} \right).
   \]
\end{proof}
\begin{thm}[Gauss- Bonnet]\label{Gauss-Bonnet}
The area of a hyperbolic \( n \)-gon \( P \) with interior angles \( \theta_1, \dots, \theta_n \) is given by the formula:  
\(
\text{Area}(P) = (n - 2)\pi - (\theta_1 + \dots + \theta_n)
\).
\end{thm}
\begin{proof}
    See \cite{AFB}
\end{proof}

\section{Optimization techniques}
In this section, we prove the following proposition using the method of Lagrange multipliers, and we establish Theorem \ref{Uniform case}.

\begin{prop}\label{Proposition}
Let $n \geq 3$ be an integer. Consider the function 
\[
F(\theta_1, \theta_2) = \cosh^{-1}\left( \frac{\cos(\pi/n)}{\sin(\theta_1/2)} \right) + \cosh^{-1}\left( \frac{\cos(\pi/n)}{\sin(\theta_2/2)} \right),
\]
defined for $(\theta_1, \theta_2) \in (0, \tfrac{(n-2)\pi}{n}) \times (0, \tfrac{(n-2)\pi}{n})$. For a fixed $r \in (0, \tfrac{2(n-2)\pi}{n})$, consider minimizing $F$ subject to the constraint $\theta_1 + \theta_2 = r$. If a minimum exists, then it must occur when $\theta_1 = \theta_2 = r/2$.
\end{prop}

\begin{proof}
Assume the minimum exists. We use the Lagrange multiplier method with the constraint function
\[
g(\theta_1, \theta_2) = \theta_1 + \theta_2 - r.
\]
At the minimum, there exists $\lambda \in \mathbb{R}$ such that $\nabla F = \lambda \nabla g$. This gives the system
\[
\frac{\partial F}{\partial \theta_1} = \lambda, \quad \frac{\partial F}{\partial \theta_2} = \lambda, \quad \theta_1 + \theta_2 = r
\]
\[
\implies
\frac{\partial F}{\partial \theta_1} = \frac{\partial F}{\partial \theta_2}.
\]
Evaluating the partial derivatives yields
\begin{align*}
-\frac{\cos(\pi/n) \cos(\theta_1/2)}{2 \sin(\theta_1/2) \sqrt{\cos^2(\pi/n) - \sin^2(\theta_1/2)}} &= -\frac{\cos(\pi/n) \cos(\theta_2/2)}{2 \sin(\theta_2/2) \sqrt{\cos^2(\pi/n) - \sin^2(\theta_2/2)}} \\[1ex]
\implies \frac{\cos(\theta_1/2)}{\sin(\theta_1/2) \sqrt{\cos^2(\pi/n) - \sin^2(\theta_1/2)}} &= \frac{\cos(\theta_2/2)}{\sin(\theta_2/2) \sqrt{\cos^2(\pi/n) - \sin^2(\theta_2/2)}} \tag{1} \\[1ex]
\implies \cos^2(\theta_1/2) \sin^2(\theta_2/2) \left( \cos^2(\pi/n) - \sin^2(\theta_2/2) \right) &= \cos^2(\theta_2/2) \sin^2(\theta_1/2) \left( \cos^2(\pi/n) - \sin^2(\theta_1/2) \right).
\end{align*}

Using the identity $2\cos A \cos B=\sin(A+B)-\sin(A-B)$ and the constraint $\theta_1 + \theta_2 = r$,
\[
\begin{aligned}
\implies \sin\left( r/2 \right) \sin\left( (\theta_1 - \theta_2)/2 \right) \Big[ &\sin^2\left( r/2 \right) + \sin^2\left( (\theta_1 - \theta_2)/2 \right) \\
&+ 2\sin^2\left( \theta_1/2 \right) + 2\sin^2\left( \theta_2/2 \right) - 4\cos^2\left( \pi/n \right) \Big] = 0.
\end{aligned}
\]
Since $r/2<\pi/2$, we know $\sin(r/2) \neq 0$, leaving us with two cases.

\vspace{0.5em}
\noindent \textbf{Case 1:} $\sin\left( (\theta_1 - \theta_2)/2 \right) = 0 \implies \theta_1 = \theta_2 = r/2$.

\vspace{0.5em} 
\noindent \textbf{Case 2:} \quad $\displaystyle \sin^2\left( r/2 \right) + \sin^2\left( (\theta_1 - \theta_2)/2 \right) + 2\sin^2\left( \theta_1/2 \right) + 2\sin^2\left( \theta_2/2 \right) - 4\cos^2\left( \pi/n \right) = 0$ \hfill (2)
\[
\implies \left[ \sin\left( r/2 \right) + \sin\left( (\theta_1 - \theta_2)/2 \right) \right]^2 = 4\left( \cos^2\left( \pi/n \right) - \sin^2\left( \theta_2/2 \right) \right) \tag{3}
\]
\[
\implies \cos\left( \theta_1/2 \right) \cos\left( \theta_2/2 \right) = \sin\left( \pi/n \right). \tag{4}
\]
Let $X=\cosh^{-1}\left(\cos(\pi/n)/\sin(\theta_1/2)\right)$ and $Y=\cosh^{-1}\left(\cos(\pi/n)/\sin(\theta_2/2)\right)$.
\[
\cosh(F(\theta_1, \theta_2)) = \cosh(X + Y) = \cosh X \cosh Y + \sinh X \sinh Y
\]
\[
\implies \cosh(F(\theta_1, \theta_2)) = \frac{\cos^2(\pi/n) + \sqrt{ \left( \cos^2(\pi/n) - \sin^2(\theta_1/2) \right) \left( \cos^2(\pi/n) - \sin^2(\theta_2/2) \right) }}{\sin(\theta_1/2) \sin(\theta_2/2)}. \tag{5}
\]
To simplify further, let $C = \cos(\pi/n)$, $S = \sin(\pi/n)$, $A = \theta_1/2$, and $B = \theta_2/2$. Then $A + B = r/2$, and from (4) we have $\cos A \cos B = S$.
\[
\cos^2(\pi/n) - \sin^2 A = C^2 - \sin^2 A = \cos^2 A - (1 - C^2) = \cos^2 A - S^2
\]
\[
\cos^2(\pi/n) - \sin^2 B = \cos^2 B - S^2
\]
\[
\implies (\cos^2 A - S^2)(\cos^2 B - S^2) = \cos^2 A \cos^2 B - S^2(\cos^2 A + \cos^2 B) + S^4.
\]
Using $\cos A \cos B = S$, we get $\cos^2 A \cos^2 B = S^2$ and $\cos^2 A + \cos^2 B = 1 + \cos(r/2) \cos(A-B)$.
\[
\implies (\cos^2 A - S^2)(\cos^2 B - S^2) = -S^2 \cos(r/2) \cos(A-B) + S^4. \tag{6}
\]
We know $\cos(A - B) = \cos A \cos B + \sin A \sin B = S + \sin A \sin B$ and $\cos(A + B) = S - \sin A \sin B \implies \sin A \sin B = S - \cos(r/2)$.
\[
\implies \cos(A - B) = 2S - \cos(r/2).
\]
Substituting into (6)
\[
(\cos^2 A - S^2)(\cos^2 B - S^2) = -S^2 \cos(r/2)(2S - \cos(r/2)) + S^4 = S^2 (\cos(r/2) - S)^2.
\]
Taking the square root with $\sin A \sin B = S - \cos(r/2) > 0$
\[
\implies \sqrt{(\cos^2 A - S^2)(\cos^2 B - S^2)} = S (S - \cos(r/2)).
\]
Substituting this into (5)
\[
\cosh(F) = \frac{C^2 + S(S - \cos(r/2))}{S - \cos(r/2)} = \frac{1 - S \cos(r/2)}{S - \cos(r/2)}. \tag{7}
\]
Evaluating at $\theta_1 = \theta_2 = r/2$
\[
F(r/2, r/2) = 2 \cosh^{-1}\left( \frac{\cos(\pi/n)}{\sin(r/4)} \right).
\]
Using $\cosh(2x)=2\cosh^2x-1$ and $\sin^2(r/4) = \frac{1 - \cos(r/2)}{2}$
\[
\implies \cosh(F(r/2, r/2)) = \frac{4 \cos^2(\pi/n)}{1 - \cos(r/2)} - 1 = \frac{3 - 4S^2 + \cos(r/2)}{1 - \cos(r/2)}. \tag{8}
\]
Let $C' = \cos(r/2)$. Considering the difference $D = \cosh(F(\theta_1, \theta_2)) - \cosh(F(r/2, r/2))$
\[
D = \frac{1 - SC'}{S - C'} - \frac{3 - 4S^2 + C'}{1 - C'} = \frac{(1 - SC')(1 - C') - (3 - 4S^2 + C')(S - C')}{(S - C')(1 - C')}.
\]
Let $N$ be the numerator
\[
N = 1 - C' - SC' + SC'^2 - 3S + 4S^3 - SC' + 3C' - 4S^2C' + C'^2
\]
\[
\implies N = (S + 1)C'^2 + (2 - 2S - 4S^2)C' + (1 - 3S + 4S^3) \tag{9}
\]
\[
\implies N = (S + 1)\left[ C'^2 + (-4S + 2)C' + (2S - 1)^2 \right] = (S + 1)\left( C' - (2S - 1) \right)^2
\]
\[
\implies D = \frac{(S + 1)(C' - (2S - 1))^2}{(S - C')(1 - C')}.
\]
Since $S + 1 > 0$, $(C' - (2S - 1))^2 \geq 0$, $S - C' > 0$, and $1 - C' > 0$, we have $D \geq 0$.
\[
\implies \cosh(F(\theta_1, \theta_2)) \geq \cosh(F(r/2, r/2))
\]
\[
\implies F(\theta_1, \theta_2) \geq F(r/2, r/2).
\]
Equality holds when $C' = 2S - 1$, which corresponds to $\theta_1 = \theta_2 = r/2$.

In Case 1, we have $\theta_1 = \theta_2 = r/2$. In Case 2, we have shown that $F(\theta_1, \theta_2) \geq F(r/2, r/2)$. Therefore, the minimum value of $F$ subject to $\theta_1 + \theta_2 = r$ is achieved at $\theta_1 = \theta_2 = r/2$.
\end{proof}

Utilizing the proposition \ref{Proposition} we prove the following theorem.

\begin{thm}\label{Uniform case}
Let $k,g\geq 2$ be integers such that $k$ divides $8g-8$ and let $n=(8g-8+4k)/k$. Let \( P_1, \dots, P_k \) be hyperbolic \(n\)-gons. Suppose  
 \[ \sum_{i=1}^{k} \text{Area}(P_i) = 2\pi(g - 2) \]  
for some integer \( g \geq 2 \). Then,  
 \[ \sum_{i=1}^{k} \text{Peri}(P_i) \geq 2nk \cosh^{-1}\left(\sqrt{2} \cos (\pi/n)\right), \] 
with equality when all $P_i$ are isometric to a regular right-angled $n$-gon.
\end{thm}

\begin{proof}
Without loss of generality, we can assume all $P_i$ are regular \cite{BK}. Let $\theta_i$ be the interior angle of $P_i$. The perimeter of a regular hyperbolic $n$-gon is given by $2n\cosh^{-1}\left(\frac{\cos (\pi/n)}{\sin(\theta_i/2)}\right)$. By the Gauss-Bonnet theorem \ref{Gauss-Bonnet}, the condition $\sum_{i=1}^{k} \text{Area}(P_i) = 2\pi(g - 2)$ is equivalent to $\sum_{i=1}^{k} \theta_i=k\pi/2$.

Let $f(\theta) = \cosh^{-1}\left(\frac{\cos(\pi/n)}{\sin(\theta/2)}\right)$ and $F(\theta_1,\dots \theta_k) = \sum_{i=1}^{k} f(\theta_i)$. We seek to minimize $\sum_{i=1}^{k}f(\theta_i)$ subject to the constraint $\sum_{i=1}^{k} \theta_i=k\pi/2$.

Observe that $\theta_i \to 0 \implies f(\theta_i) \to \infty$. Thus, the minima cannot be attained in a neighborhood where any $\theta_i$ is close to $0$. There must exist $\epsilon > 0$ such that $\theta_i \geq \epsilon$ for all $1\leq i \leq k$, which implies $\theta_i\leq k\pi/2-(k-1)\epsilon$. We minimize $F$ over the domain $D=\{(\theta_1,\dots,\theta_k)\mid \epsilon \leq \theta_i \leq k\pi/2-(k-1)\epsilon \}$. Since $D$ is a compact set in $\mathbb{R}^k$ and $F$ is a continuous function, the minimum of $F$ exists.
 
Suppose the minimum is attained at $(\theta_1',\dots, \theta_k')$ where not all $\theta_i'$ are equal. Without loss of generality, assume $\theta_1'\neq \theta_2'$. Let $\theta_1'+ \theta_2'=r$, where $r=k\pi/2-\sum_{i=3}^{k}\theta_i'$. By Proposition \ref{Proposition}, we have
\[ 2f(r/2) + \sum_{i=3}^{k}f(\theta_i') < \sum_{i=1}^{k}f(\theta_i'), \]
contradicting that the minimum is attained at $(\theta_1',\dots, \theta_k')$. Thus, all $\theta_i$ must be equal. 

\[ \theta_i = \theta_j \implies \theta_i=\pi/2 \text{ for all } 1\leq i \leq k. \]
This means the minimum is attained when all $P_i$ are regular right-angled $n$-gons.
\end{proof}

\section{Uniform case}
In this section, as a direct application of Theorem \ref{Uniform case}, we provide an elementary calculation of the exact minima of the geodesic length functions for the uniform case. While these minima were recently computed in \cite{GEGGHR} using heavy algebraic machinery, our approach independently recovers these results for $4$-regular topological filling and identifies exactly where they are realized.

\begin{thm}
    Let \(\Omega = \{\gamma_1, \ldots, \gamma_d\}\) be a uniform filling on $S_g$. For a hyperbolic surface $X$, suppose \(S_g \setminus \Omega\) is a disjoint union of \(k\) regular \(n\)-gons. Then,
    \[ \ell_\Omega(X) \geq 2nk \cosh^{-1}\left(\sqrt{2} \cos(\pi/n)\right). \]
    Equality holds if and only if \(X\) is obtained by a side-pairing of \(k\) regular, right-angled hyperbolic \(n\)-gons.
\end{thm}

\begin{proof}
Define the sets:
\[
\mathcal{A} = \left\{ \sum_{i=1}^{k} \ell(P_i) \;\middle|\; 
\begin{array}{l}
P_1, \dots, P_k \text{ are $n$-gons that admit} \\
\text{a side-pairing yielding a hyperbolic surface}
\end{array} \right\},
\]
\[
\mathcal{B} = \left\{ \sum_{i=1}^{k} \ell(P_i) \;\middle|\; P_1, \dots, P_k \text{ are $n$-gons} \right\}.
\]
Clearly $\mathcal{A} \subseteq \mathcal{B} \implies \min \mathcal{A} \geq \min \mathcal{B}$.

By Theorem \ref{Uniform case}, the minimum over $\mathcal{B}$ is attained when all $n$-gons are regular and right-angled. Moreover, such a configuration belongs to $\mathcal{A}$ since these polygons can be side-paired to form a hyperbolic surface. 
\[ \implies \min \mathcal{A} = \min \mathcal{B}. \]
Thus, the minimum is achieved when all $n$-gons are regular right-angled.
\end{proof}
 
\section{Non-uniform case}
In this section, we describe two types of non-uniform fillings and explicitly calculate their respective minima for the geodesic length function. We establish that both types realize their minima at a triangle surface. The following proposition describes the first type of non-uniform filling. 

\begin{prop}\label{Prop-non-uniform filling}
    There exists a filling \(\Omega_g\) on \(S_{g}\) such that $S_g\setminus \Omega_g=\bigcup_{i=1}^3 P_i$, and
\[ |P_{1}| = 2g+1,\quad |P_{2}| = 2g+1,\quad |P_{3}| = 4g+2. \]
\end{prop}

\begin{proof}
Let the set of directed edges be $E = \{e_1, \dots, e_{2g+1}, f_1, \dots, f_{2g+1}, e_1', \dots, e_{2g+1}', f_1', \dots, f_{2g+1}'\}$. Consider a graph $\Gamma$ with vertices $V=E/\sim$ and edges $E/\sigma_1$. We define a fat graph structure on $\Gamma$ via the vertex permutation $\sigma_0$, consisting of $2g+1$ vertices. For $i = 1, \dots, 2g+1$ (with indices taken modulo $2g+1$), the vertices are given by:
\[ \sigma_0 = \prod_{i=1}^{2g+1} (e_i, f_i, f_{i+1}', e_{i-1}'). \]

By computing $\sigma_0 \circ \sigma_1$, we obtain three boundary components of $\Gamma$.
\begin{align*}
\sigma_0\circ\sigma_1 = & (e_1, e_2, \dots, e_{2g+1}) \\
                        & (f_1', f_2', \dots, f_{2g+1}') \\
                        & (e_1', f_1, e_{2g}', f_{2g}, e_{2g-2}', f_{2g-2}, \dots, e_{3}', f_{3}).
\end{align*}

The graph $\Gamma$ has $|V| = 2g+1$ vertices, $|E| = 4g+2$ undirected edges, and $|F| = 3$ faces. From Euler's characteristic formula, we calculate the genus of the fat graph as follows:
\[ |V| - |E| + |F| = (2g+1) - (4g+2) + 3 = 2 - 2g. \]

Furthermore, there are exactly $2g+1$ standard cycles for the graph $\Gamma$, computed from $\sigma_0^2\circ \sigma_1$:
\[ \sigma_0^2\circ \sigma_1 = \prod_{i=1}^{2g+1} (e_i, f_{i+1})(e_i', f_{i+1}'). \]
\end{proof}

\begin{eg}\label{Example-1 non-uniform filling}
To visually illustrate the construction in Proposition \ref{Prop-non-uniform filling}, consider the specific case of the genus two surface ($g=2$). The complement $S_2 \setminus \Omega_2$ decomposes into two pentagons ($|P_1|=|P_2|=5$) and one decagon ($|P_3|=10$). 

\begin{rmk}
While fat graphs are formally combinatorial objects, we annotate Figure \ref{Fatgraph_Non-uniform} with geometric angles ($\theta_i$ and $\pi - \theta_i$) to support our analysis. This visually clarifies how each vertex of the graph forms a corner in the resulting hyperbolic polygons.
\end{rmk} 

The corresponding fat graph $\Gamma$, consisting of $5$ vertices and $10$ edges, is depicted in Figure \ref{Fatgraph_Non-uniform}, with its three geometric boundary components explicitly drawn in Figure \ref{3_polygon}.

    \begin{figure}[H]
        \centering
        \includegraphics[width=0.999\linewidth]{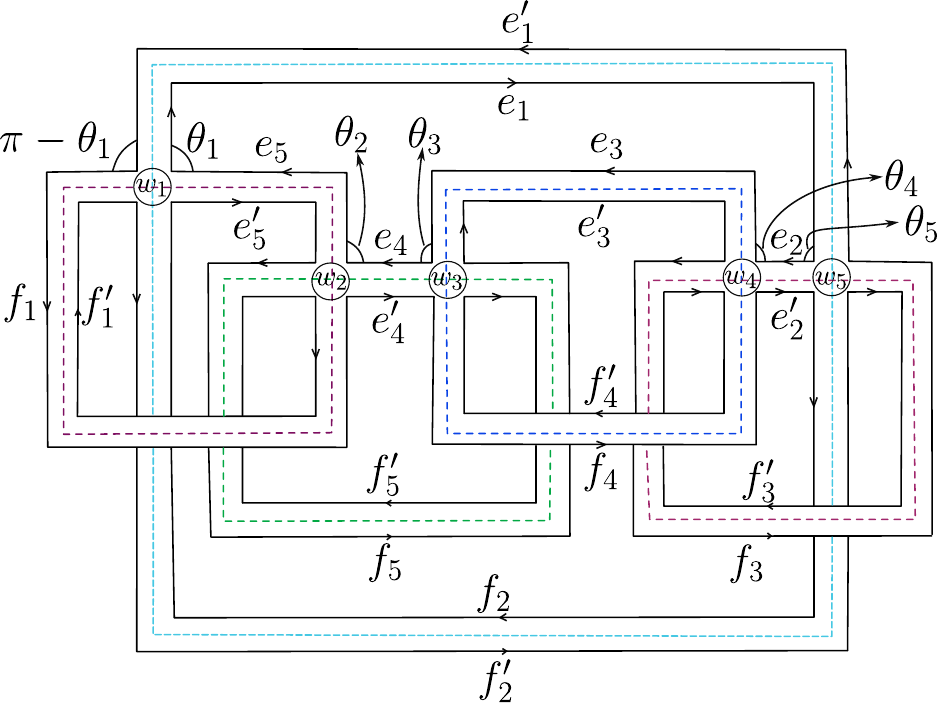}
        \caption{Fat graph, annotated with geometric angles.}
        \label{Fatgraph_Non-uniform}
    \end{figure}

    \begin{figure}[H]
        \centering
        \includegraphics[width=0.8\linewidth]{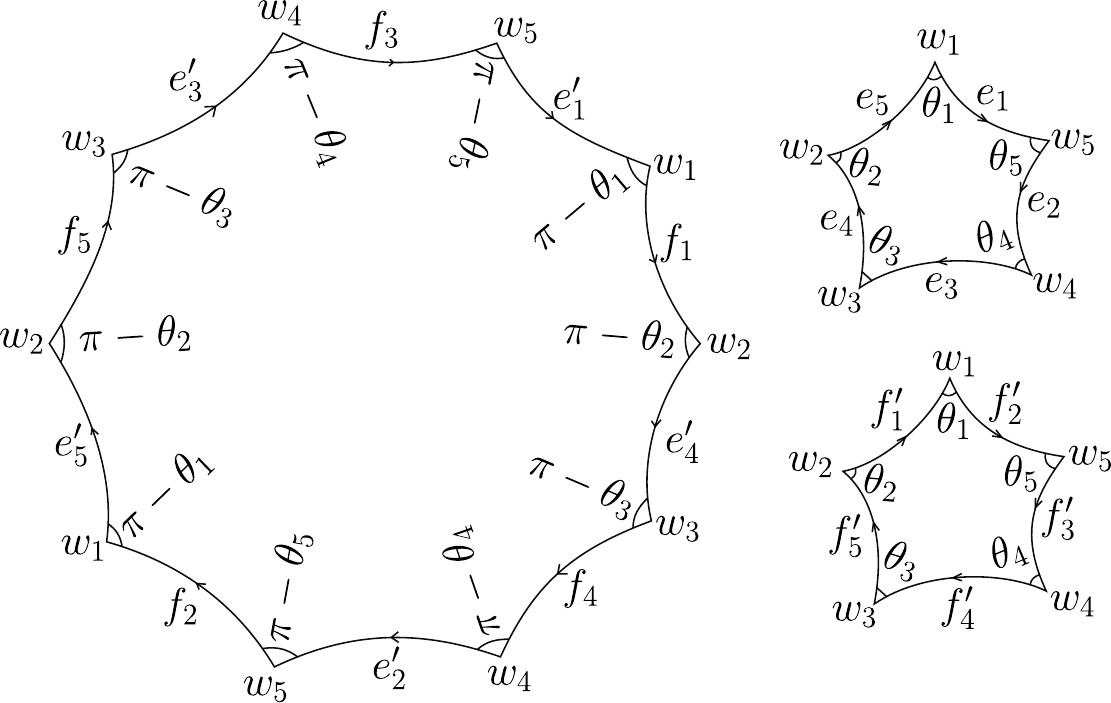}
        \caption{The three boundary components.}
        \label{3_polygon}
    \end{figure}
\end{eg}

\begin{lemma}\label{lem:supplementary_angles}
Let $P_1$ be a regular hyperbolic $n$-gon with interior angle $\theta$, and let $P_2$ be a regular hyperbolic $m$-gon with interior angle $\pi - \theta$. If $P_1$ and $P_2$ have the same side length, then $\theta$ is uniquely determined by the formula
\[ \tan(\theta/2) = \frac{\cos(\pi/n)}{\cos(\pi/m)}. \]
\end{lemma}

\begin{proof}
Let $\ell$ be the common side length. Applying Lemma \ref{RightAngledTriangle} to $P_1$ and $P_2$ yields
\[ \cos(\pi/n) = \cosh(\ell/2) \sin(\theta/2) \quad \text{and} \quad \cos(\pi/m) = \cosh(\ell/2) \sin((\pi - \theta)/2) \]
\[ \implies \frac{\cos(\pi/n)}{\cos(\pi/m)} = \tan(\theta/2). \]
Since $\tan(\theta/2)$ is strictly increasing for $\theta \in (0, \pi) \implies \theta$ is uniquely determined.
\end{proof}

The following lemma establishes a general geometric bound that will be the core step in determining the minimal length of our filling.

\begin{lemma}\label{Lemma-Perimeter-Minimization-General}
Let $Q_1, \dots, Q_k$ be hyperbolic $n$-gons such that the sum of all their $kn$ interior angles is $s$. Let $Q^*$ be the regular hyperbolic $n$-gon with interior angle $s/kn$. Then
\[ \sum_{i=1}^k \text{Peri}(Q_i) \ge k \cdot \text{Peri}(Q^*). \]
Equality holds if and only if each $Q_i$ is regular and isometric to $Q^*$.
\end{lemma}

\begin{proof}
For each $i \in \{1, \dots, k\}$, let $s_i$ be the sum of the $n$ interior angles of $Q_i$, and let $Q_i^{reg}$ be the corresponding regular $n$-gon with interior angle $\theta_i = s_i/n$. By Bezdek's theorem \cite{BK}, $\text{Peri}(Q_i) \ge \text{Peri}(Q_i^{reg})$, with equality if and only if $Q_i$ is isometric to $Q_i^{reg}$. Thus,
\[ \sum_{i=1}^k \text{Peri}(Q_i) \ge \sum_{i=1}^k \text{Peri}(Q_i^{reg}). \]

If any two interior angles $\theta_i \neq \theta_j$, Proposition \ref{Proposition} implies that replacing both $Q_i^{reg}$ and $Q_j^{reg}$ with two identical regular $n$-gons of angle $(\theta_i + \theta_j)/2$ decreases the total perimeter while preserving the total angle sum. Since the total angle sum across all $k$ polygons is fixed and $s= \sum_{i=1}^k n \theta_i$, a repeated pairwise application of Proposition \ref{Proposition} implies that the sum is minimized when all $\theta_i$ are identical and equal to the average $s/kn$. This configuration exactly corresponds to $k$ copies of $Q^*$.
\end{proof}

In the following theorem, we explicitly calculate the minima of the geodesic length functions for the filling $\Omega_g$ defined in Proposition \ref{Prop-non-uniform filling}. Furthermore, we identify the surface where these minima are realized.

\begin{thm}\label{Geodesic length minima for Non-uniform case}
For the filling $\Omega_g$ on $S_g$,
$$ \ell_{\Omega_g}(X) \ge (2g+1)\Bigg[ \cosh^{-1}\!\left(\frac{\cos(\pi/(2g+1))}{\sin(\theta/2)}\right) + \cosh^{-1}\!\left(\frac{\cos(\pi/(4g+2))}{\sin(\theta/2)}\right)\Bigg], $$
where $\theta$ satisfies
$$ \tan\left(\frac{\theta}{2}\right) = \frac{\cos(\pi/(4g+2))}{\cos(\pi/(2g+1))}. $$
Equality holds if and only if $X$ is the triangle surface of type $(2, 2g+1, 4g+2)$.
\end{thm}

\begin{proof}
By Kerckhoff's theorem \cite{KSP}, the length function $\ell_{\Omega_g}$ admits a unique global minimum in $\mathcal{T}_g$. Let $X_0$ be this minimizing surface.

Let $P_1, P_2$ be the two $(2g+1)$-gons and $P_3$ be the $(4g+2)$-gon of the filling on $X_0$. Let $\theta_1, \dots, \theta_{2g+1}$ be the interior angles of $P_1$ and $P_2$. The interior angles of $P_3$ are the corresponding supplementary angles $\pi - \theta_k$. The total filling length is $\ell_{\Omega_g}(X_0) = \frac{1}{2} \big( {Peri}(P_1) + {Peri}(P_2) + {Peri}(P_3) \big)$.

By Lemma \ref{Lemma-Perimeter-Minimization-General}, the perimeter sum of $P_1$ and $P_2$ is minimized by two identical regular $(2g+1)$-gons, denoted $P^*$, with an interior angle of $\theta_0 = \frac{1}{2g+1}\sum_{k=1}^{2g+1}\theta_k$. By Bezdek's theorem \cite{BK}, the perimeter of $P_3$ is minimized by a regular $(4g+2)$-gon $Q^*$ with an interior angle of $\pi - \theta_0$. Thus, we have:
$$ \ell_{\Omega_g}(X_0) \ge {Peri}(P^*) + \frac{1}{2} {Peri}(Q^*). $$

We now prove that this lower bound is realized by a filling on $S_g$ for a suitable choice of $\theta_0$. To realize this minimum, we assign the regular polygons $P^*$ and $Q^*$ to the three complementary regions of $\Omega_g$. The geometric side-pairing conditions of the filling require the side lengths of $P^*$ and $Q^*$ to be equal. By Lemma \ref{lem:supplementary_angles}, this condition uniquely determines the interior angle $\theta_0$, satisfying:
$$ \tan\left(\theta_0/2\right) = \frac{\cos(\pi/(4g+2))}{\cos(\pi/(2g+1))}. $$

Substituting the corresponding regular side lengths into the perimeter sum explicitly yields:
$$ \ell_{\Omega_g}(X) \ge (2g+1)\Bigg[ \cosh^{-1}\!\left(\frac{\cos(\pi/(2g+1))}{\sin(\theta/2)}\right) + \cosh^{-1}\!\left(\frac{\cos(\pi/(4g+2))}{\sin(\theta/2)}\right)\Bigg]. $$

Finally, we identify the minimizing surface $X_0$. Because the minimal configuration consists entirely of regular polygons with equal side lengths, a "one-click" rotational shift $\phi: w_k \mapsto w_{k+1}$ of the regular polygon $Q^*$ induces a global isometry of order $4g+2$ on $X_0$. 
The quotient $X_0 / \langle \phi \rangle$ has exactly three cone points: 
The center of $Q^*$ with order $4g+2$. 
The vertices $w_k$ with order $2$. 
The centers of $P_1$ and $P_2$ with order $2g+1$.

Thus, the quotient $X_0 / \langle \phi \rangle$ has signature $(0; 2, 2g+1, 4g+2)$. This proves that the minimum is attained exactly on the triangle surface of type $(2, 2g+1, 4g+2)$.
\end{proof}

The following proposition introduces a second type of non-uniform filling for any surface of genus $g \ge 2$. Following the proof, we provide a detailed visual example for the genus two case.

\begin{prop}\label{Prop-general-genus-filling-g2}
For $g \ge 2$, there exists a filling $\Omega_g'$ on $S_g$ such that $S_g \setminus \Omega_g' = \bigcup_{i=1}^{2g+2} Q_i$, with
$$|Q_{1}| = \dots = |Q_{2g}| = 4, \quad \text{and} \quad |Q_{2g+1}| = |Q_{2g+2}| = 4g.$$
\end{prop}

\begin{proof}
Let $\Gamma_g$ be the $4$-regular graph having $4g$ vertices and an edge set $E = \{e_i, e_i', f_i, f_i' \mid 1 \le i \le 4g\}$.

The fat graph structure is defined by the following permutation, where all indices are calculated modulo $4g$:
$$\sigma_0 = \prod_{i=1}^{4g} (e_i, e_{i-1}', f_i, f_{i-2g+1}')$$

Upon computation, we obtain:
$$\sigma_0 \sigma_1 = \left( \prod_{i=1}^{2g} (e_i, f_{i+1}, e_{i+2g}, f_{i+2g+1}) \right) (e_1', e_{4g}', e_{4g-1}', \dots, e_2') (f_{c_1}', f_{c_2}', \dots, f_{c_{4g}}')$$
where $c_k \equiv 2g + 1 - (k-1)(2g-1) \pmod{4g}$.
This indicates that there are $2g$ boundary components having a combinatorial length of $4$, and $2$ components having a combinatorial length of $4g$. 

Let $Q_k$, for $k = 1, \dots, 2g$, be the quadrilaterals defined by the vertex sequence $(w_k, w_{k+1}, \allowbreak w_{k+2g}, w_{k+2g+1})$ and the boundary edge cycle $(e_k, f_{k+1}, \allowbreak e_{k+2g}, f_{k+2g+1})$, with interior angles $\theta_i$ located at each respective vertex $w_i$. Here, $e_k$ denotes the directed edge $(w_k, w_{k+1})$ and $f_k$ denotes the directed edge $(w_k, w_{k+2g-1})$.

Similarly, let $Q_{2g+1}$ and $Q_{2g+2}$ be the $4g$-gonal faces defined by the vertex sequences $(w_1, w_{4g}, w_{4g-1}, \dots, w_2)$ and $(w_{v_1}, w_{v_2}, \dots, w_{v_{4g}})$ respectively. These $4g$-gons are bounded by the edge cycles $(e_1', e_{4g}', e_{4g-1}', \dots, e_2')$ and $(f_{c_1}', f_{c_2}', \dots, f_{c_{4g}}')$, with interior angles given by the supplements $\bar{\theta}_i = \pi - \theta_i$ located at each respective vertex $w_i$. Here, $e_k'$ denotes the directed edge $(w_{k+1}, w_k)$ and $f_k'$ denotes the directed edge $(w_{k+2g-1}, w_k)$.

All indices are calculated modulo $4g$. The full description of all the polygons is provided in Table \ref{table_general}. All subscripts are taken modulo $4g$.

\begin{table}[H]
\centering
\caption{Geometric specifications of the complementary polygons for genus $g \ge 2$.} 
\label{table_general}
\resizebox{\textwidth}{!}{%
\begin{tabular}{@{}llll@{}}
\toprule
Face & Vertex Sequence & Boundary Edge Cycle & Interior Angles \\ \midrule
$Q_k \ (1 \le k \le 2g)$ 
& $(w_k, w_{k+1}, w_{k+2g}, w_{k+2g+1})$ 
& $(e_k, f_{k+1}, e_{k+2g}, f_{k+2g+1})$ 
& $(\theta_k, \theta_{k+1}, \theta_{k+2g}, \theta_{k+2g+1})$ \\ \addlinespace

$Q_{2g+1}$ 
& $(w_1, w_{4g}, w_{4g-1}, \dots, w_2)$ 
& $(e_{4g}', e_{4g-1}', \dots, e_2', e_1')$ 
& $(\bar{\theta}_1, \bar{\theta}_{4g}, \bar{\theta}_{4g-1}, \dots, \bar{\theta}_2)$ \\ \addlinespace

$Q_{2g+2}$ 
& $(w_1, w_{1-(2g-1)}, w_{1-2(2g-1)}, \dots, w_{2g})$ 
& $(f_{1-(2g-1)}', f_{1-2(2g-1)}', \dots, f_{2g}', f_1')$ 
& $(\bar{\theta}_1, \bar{\theta}_{1-(2g-1)}, \bar{\theta}_{1-2(2g-1)}, \dots, \bar{\theta}_{2g})$ \\ \bottomrule
\end{tabular}%
}
\end{table}

Tracing the edge mappings under $\sigma_0^2 \sigma_1$ yields $e_k \mapsto f_{k-d}' \mapsto e_{k-d}$ and $f_k \mapsto e_{k+d}' \mapsto f_{k+d}$, with indices taken modulo $4g$ and step size $d = 2g - 2$. This partitions the edges into two disjoint families. Letting $c = \gcd(4g, 2g-2) = 2\gcd(2, g-1)$, the disjoint cycle representation is:
$$\sigma_0^2 \sigma_1 = \prod_{j=1}^{c} ( e_j, f_{j-d}', e_{j-d}, f_{j-2d}', e_{j-2d}, \dots ) \prod_{j=1}^{c} ( f_j, e_{j+d}', f_{j+d}, e_{j+2d}', f_{j+2d}, \dots )$$

Each family contains $c$ directed, edge-disjoint cycles, totaling $4\gcd(2, g-1)$ cycles. Standard cycles are defined as equivalence classes under the boundary reversal map $C = (x_1, \dots, x_m) \sim C^{-1} = (x_m', \dots, x_1')$. Because this reversal acts as a bijection between the two families, the number of standard cycles is exactly half the total directed cycles, yielding $2\gcd(2, g-1)$. Consequently, there are exactly $2$ standard cycles for even $g$, and $4$ for odd $g$. 

For even $g$, we obtain a required filling of size $2$, and for odd $g$, we obtain a filling of size $4$. Consequently, we obtain the filling $\Omega_g'$, and the complement of this filling decomposes the surface precisely into our defined faces, giving $S_g \setminus \Omega_g' = \bigcup_{i=1}^{2g+2} Q_i$.
\end{proof}

\begin{eg}\label{Example-2 non-uniform filling}
To make the abstract construction in Proposition \ref{Prop-general-genus-filling-g2} concrete, consider the case for $g=2$. Here, the filling $\Omega_2'$ decomposes the surface into four quadrilaterals ($|Q_1|=|Q_2|=|Q_3|=|Q_4|=4$) and two octagons ($|Q_5|=|Q_6|=8$). 

The corresponding fat graph $\Gamma_2$ is a $4$-regular graph with $8$ vertices and $16$ edges. Its vertex permutation $\sigma_0$ expands explicitly to:
$$\begin{aligned}
\sigma_0 &= (e_1, e_8', f_1, f_6')(e_2, e_1', f_2, f_7')(e_3, e_2', f_3, f_8')(e_4, e_3', f_4, f_1') \\
&\qquad (e_5, e_4', f_5, f_2')(e_6, e_5', f_6, f_3')(e_7, e_6', f_7, f_4')(e_8, e_7', f_8, f_5').
\end{aligned}$$

The six boundary components, determined by $\sigma_0 \circ \sigma_1$, match exactly with the faces detailed in Table \ref{table_g2} below. 

$$
\begin{aligned}
\sigma_0\sigma_1 &= (e_1, f_2, e_5, f_6)(e_2, f_3, e_6, f_7)(e_3, f_4, e_7, f_8)(e_4, f_5, e_8, f_1) \\
&\qquad (e_1', e_8', e_7', e_6', e_5', e_4', e_3', e_2')(f_1', f_6', f_3', f_8', f_5', f_2', f_7', f_4')
\end{aligned}
$$
Furthermore, computing $\sigma_0^2 \sigma_1$ reveals the standard cycles that form the filling. The computation yields:
$$
\begin{aligned}
\sigma_0^2\sigma_1 &= (e_1, f_7', e_7, f_5', e_5, f_3', e_3, f_1')(e_2, f_8', e_8, f_6', e_6, f_4', e_4, f_2') \\
&\qquad (f_1, e_3', f_3, e_5', f_5, e_7', f_7, e_1')(f_2, e_4', f_4, e_6', f_6, e_8', f_8, e_2')
\end{aligned}
$$

% $$\begin{aligned}
% \sigma_0^2 \sigma_1 &= (e_1, f_5', e_2, f_5') \\
% &\qquad (e_2, f_5', e_8, f_6', e_6, f_4', e_4, f_2') \\
% &\qquad (f_5, e_1', f_7, e_1', f_1, e_3', f_3, e_5') \\
% &\qquad (f_8, e_2', f_2, e_4', f_4, e_6', f_6, e_8').
% \end{aligned}$$
From these, we identify two edge-disjoint standard cycles. These project to two curves on $S_2$, which we denote by $\gamma_1$ and $\gamma_2$, explicitly giving us the filling $\Omega_2' = \{\gamma_1, \gamma_2\}$.

\begin{table}[H]
\centering
\caption{q              specifications of the complementary polygons for genus $g=2$.}
\label{table_g2}
\begin{tabular}{@{}llll@{}}
\toprule
Face & Vertex Sequence & Boundary Edge Cycle & Interior Angles \\ \midrule
$Q_1$ & $(w_1, w_2, w_5, w_6)$ & $(e_1, f_2, e_5, f_6)$ & $(\theta_1, \theta_2, \theta_5, \theta_6)$ \\
$Q_2$ & $(w_2, w_3, w_6, w_7)$ & $(e_2, f_3, e_6, f_7)$ & $(\theta_2, \theta_3, \theta_6, \theta_7)$ \\
$Q_3$ & $(w_3, w_4, w_7, w_8)$ & $(e_3, f_4, e_7, f_8)$ & $(\theta_3, \theta_4, \theta_7, \theta_8)$ \\
$Q_4$ & $(w_4, w_5, w_8, w_1)$ & $(e_4, f_5, e_8, f_1)$ & $(\theta_4, \theta_5, \theta_8, \theta_1)$ \\ \addlinespace
$Q_5$ & $(w_1, w_8, w_7, w_6, w_5, w_4, w_3, w_2)$ & $(e_8', e_7', e_6', e_5', e_4', e_3', e_2', e_1')$ & $(\bar{\theta}_1, \bar{\theta}_8, \bar{\theta}_7, \bar{\theta}_6, \bar{\theta}_5, \bar{\theta}_4, \bar{\theta}_3, \bar{\theta}_2)$ \\
$Q_6$ & $(w_1, w_6, w_3, w_8, w_5, w_2, w_7, w_4)$ & $(f_6', f_3', f_8', f_5', f_2', f_7', f_4', f_1')$ & $(\bar{\theta}_1, \bar{\theta}_6, \bar{\theta}_3, \bar{\theta}_8, \bar{\theta}_5, \bar{\theta}_2, \bar{\theta}_7, \bar{\theta}_4)$ \\ \bottomrule
\end{tabular}
\end{table}

Visually, the local structure at the vertices of $\Gamma_2$ is shown in Figure \ref{Fatgraph_Non-uniform_figure_2}, while the six global polygonal boundary components are illustrated in Figure \ref{6_polygon}.
\begin{figure}[H]
    \centering
    \includegraphics[width=0.98\linewidth]{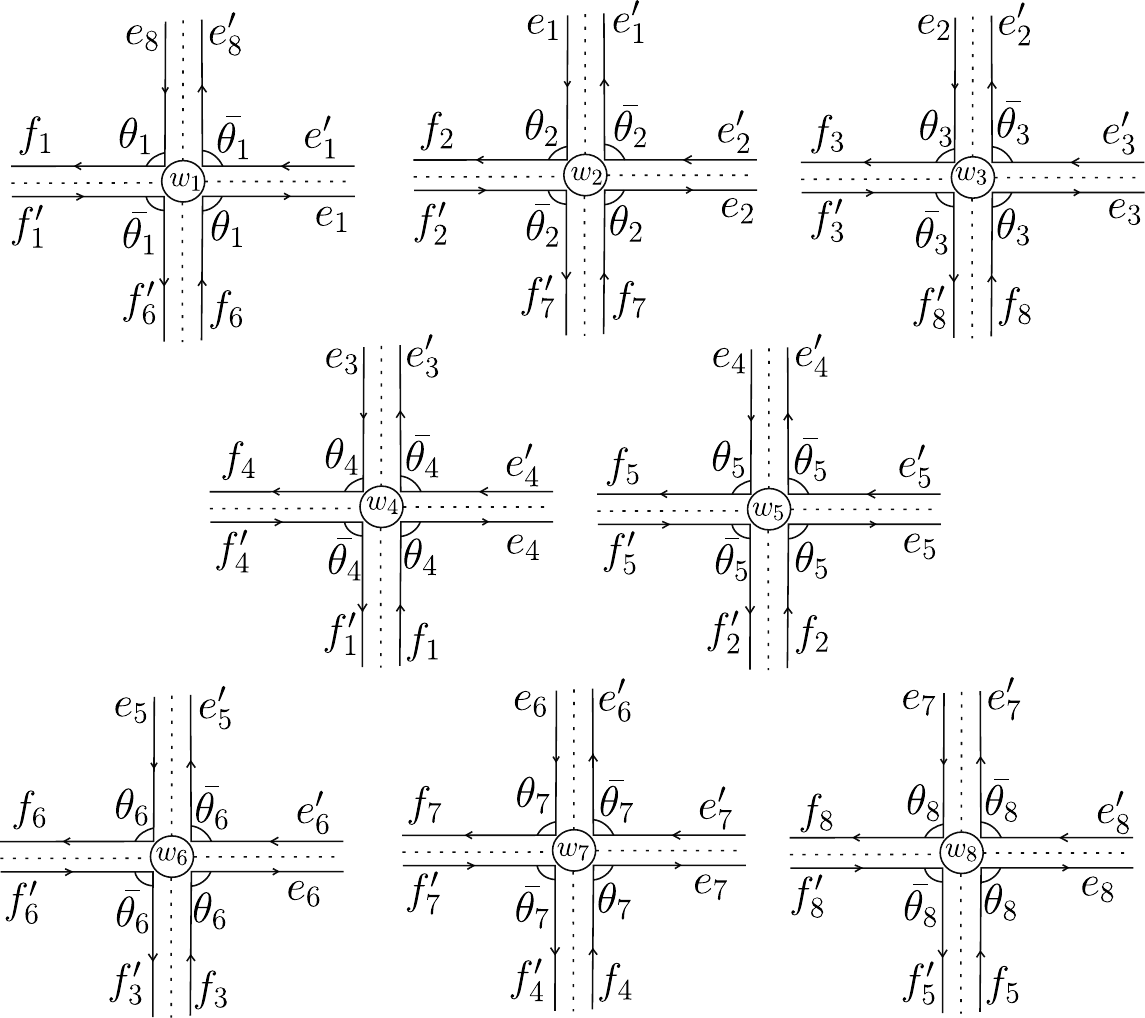}
    \caption{Local structure of a fat graph at a vertex with annotated geometric angles.}
    \label{Fatgraph_Non-uniform_figure_2}
\end{figure}

\begin{figure}[H]
    \centering
    \includegraphics[width=0.9\linewidth]{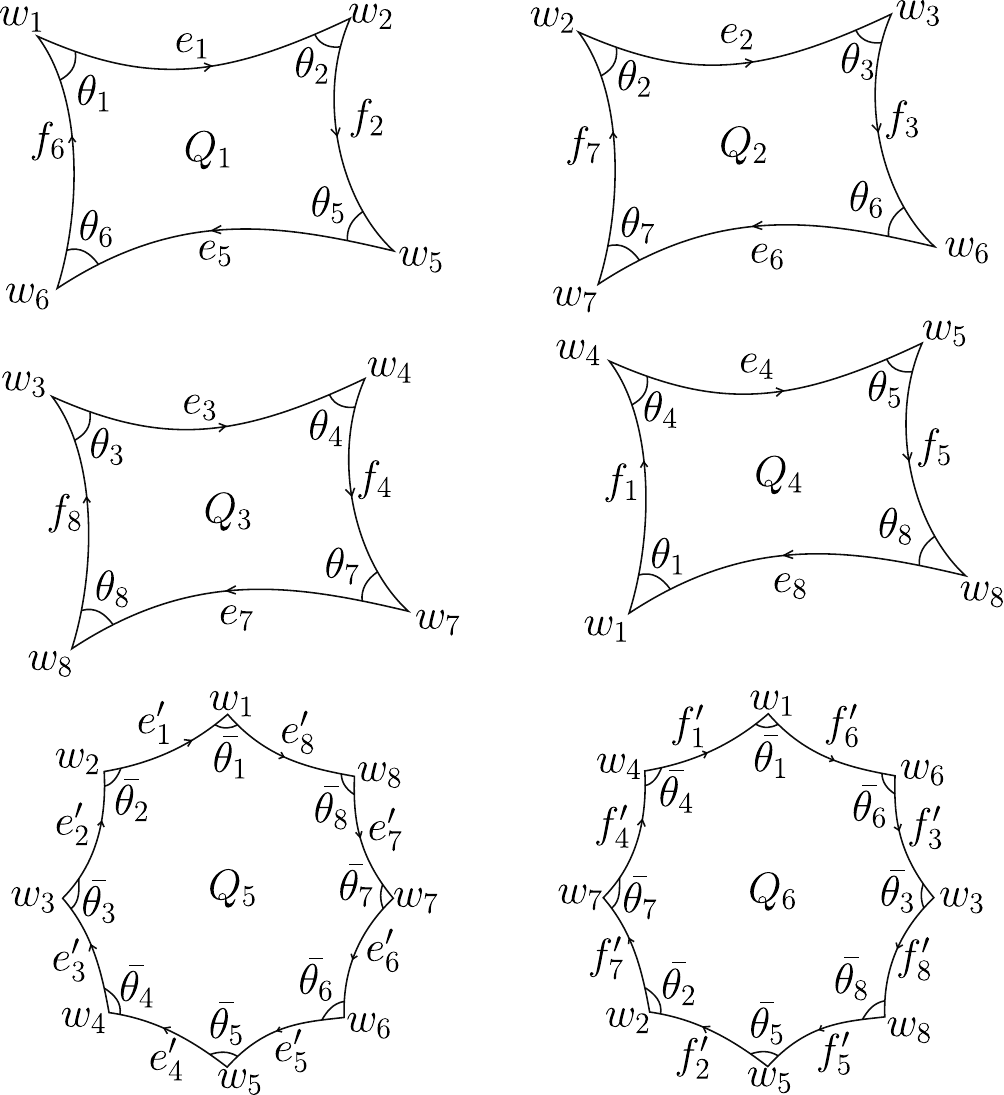}
    \caption{The six boundary components.}
    \label{6_polygon}
\end{figure}
\end{eg}

The following theorem establishes the absolute minimum of the geodesic length functions for the filling $\Omega_g'$ on $S_g$ for any $g \ge 2$.

\begin{thm}\label{Thm-Geodesic-length-minima}
Let $g \ge 2$ and $\Omega_g'$ be the filling on $S_g$ as defined in Proposition \ref{Prop-general-genus-filling-g2}. Then
$$ \ell_{\Omega_g'}(X) \ge 8g\Bigg[ \cosh^{-1}\!\left(\frac{\cos(\pi/2g)}{\sin(\theta/2)}\right) + \cosh^{-1}\!\left(\frac{\cos(\pi/4g)}{\cos(\theta/2)}\right)\Bigg], $$
where $\theta$ satisfies
$$ \tan\left(\frac{\theta}{2}\right) = \frac{\cos(\pi/4g)}{\cos(\pi/2g)}. $$
Equality holds if and only if $X$ is a Bolza-like surface (which coincides with the Bolza surface when $g=2$).
\end{thm}

\begin{proof}
Since $\Omega_g'$ is a filling on $S_g$, classical Teichmüller theory guarantees that its length function $\ell_{\Omega_g'}$ admits a unique global minimum in $\mathcal{T}_g$. Let $X_0$ be this unique minimizing hyperbolic surface, and let $\theta_1^0, \dots, \theta_{4g}^0$ be the associated interior angles formed by the boundary components of the filling $\Omega_g'$. We denote the corresponding minimizing polygons by $Q_1^0, \dots, Q_{2g+2}^0$, where $Q_1^0, \dots, Q_{2g}^0$ are quadrilaterals and $Q_{2g+1}^0, Q_{2g+2}^0$ are $4g$-gons. 

The length of the filling on $X_0$ is given by $\ell_{\Omega_g'}(X_0) = \frac{1}{2} \sum_{k=1}^{2g+2} {Peri}(Q_k^0)$.

By Lemma \ref{Lemma-Perimeter-Minimization-General}, the total perimeter of the $2g$ quadrilaterals $Q_1^0, \dots, Q_{2g}^0$ is minimized when they are all identical regular quadrilaterals. Letting $Q^*$ be the regular quadrilateral realizing their unified average interior angle, which we denote here by $s_0/4g$, we obtain:
$$ \sum_{k=1}^{2g} {Peri}(Q_k^0) \ge 2g \cdot {Peri}(Q^*). $$

For the two $4g$-gons, $Q_{2g+1}^0$ and $Q_{2g+2}^0$, their interior angles consist exactly of the supplementary angles $\pi - \theta_k^0$ for $k=1, \dots, 4g$. Letting $P^*$ be the regular $4g$-gon with an interior angle of $\pi - s_0/4g$, Bezdek's result \cite{BK} yields ${Peri}(Q_{2g+i}^0) \ge {Peri}(P^*)$ for $1\leq i\le 2$. 

Thus, we have $\ell_{\Omega_g'}(X_0) \ge g \cdot {Peri}(Q^*) + {Peri}(P^*)$. We now prove that this lower bound is realized by a filling on $S_g$ for a suitable choice of $s_0$.

To realize this minimum, the geometric side-pairing conditions require the side lengths of $Q^*$ and $P^*$ to be equal. By Lemma \ref{lem:supplementary_angles}, this condition uniquely determines the interior angle $s_0/4g$, satisfying:
$$\tan\left(\frac{s_0}{8g}\right) = \frac{\cos(\pi/4g)}{\cos(\pi/2g)}.$$

Finally, we identify the minimizing surface $X_0$. The combinatorial cyclic shift $\phi: w_k \mapsto w_{k+1} \pmod{4g}$ induces a global isometric automorphism of order $4g$ on $X_0$. In the quotient $X_0 / \langle \phi \rangle$, the geometric centers of the two $4g$-gons yield two cone points of order $4g$, and the geometric centers of the $2g$ quadrilaterals yield a single cone point of order $2$. Therefore, the quotient $X_0 / \langle \phi \rangle$ has signature $(0; 2, 4g, 4g)$, proving that $X_0$ is a Bolza-like surface.
\end{proof}

\begin{rmk}
The minimizing surface obtained in Theorem \ref{Geodesic length minima for Non-uniform case} admits the highest possible order of symmetry, namely $4g+2$, whereas the Bolza-like surface obtained in Theorem \ref{Thm-Geodesic-length-minima} realizes the second highest order of symmetry, namely $4g$.
\end{rmk}
\begin{rmk}
 There is an embedded graph on the surface, and the graph isomorphisms extend to elements of the mapping class group which leave the geodesic length function invariant. By Theorem $37$ of the article \cite{SS99}, an isolated fixed point of the mapping class group is always a critical point. Since the geodesic length function has only one critical point, it must coincide with this point with automorphism group given by a finite quotient of a triangle group. Also, Gao and Wang studied the silimar kind of results in the article \cite{gao2025}.
\end{rmk}

\section*{Acknowledgements} 
The First author would like to thank the Council of Scientific and Industrial Research (File Number: 09/1290(0005)2020-EMR-I ) for providing financial support. The second author gratefully acknowledges the financial support from the Science and Engineering Research Board (SERB), Government of India through MATRICS grant (File Number: MTR/2021/000067).

\bibliographystyle{alpha}
\bibliography{Bibilography}

\end{document}